\documentclass[12pt,leqno,a4paper]{amsart}
\usepackage{latexsym}
\usepackage{amsmath}
\usepackage{amsthm}
\usepackage{amssymb}
\usepackage{amsfonts}

\usepackage{mathrsfs}
\usepackage{hyperref}

\usepackage{comment}

\begin{document}
\theoremstyle{plain}
\newtheorem{MainThm}{Theorem}
\newtheorem{thm}{Theorem}[section]
\newtheorem{clry}[thm]{Corollary}
\newtheorem{prop}[thm]{Proposition}
\newtheorem{lemma}[thm]{Lemma}
\newtheorem{deft}[thm]{Definition}
\newtheorem{hyp}{Assumption}
\newtheorem*{conjecture}{Conjecture}
\newtheorem{question}{Question}

\newtheorem{claim}[thm]{Claim}

\theoremstyle{definition}
\newtheorem*{definition}{Definition}
\newtheorem{assumption}{Assumption}
\newtheorem{rem}[thm]{Remark}
\newtheorem*{remark}{Remark}
\newtheorem*{acknow}{Acknowledgements}

\newtheorem{example}[thm]{Example}
\newtheorem*{examplenonum}{Example}
\numberwithin{equation}{section}
\newcommand{\nocontentsline}[3]{}
\newcommand{\tocless}[2]{\bgroup\let\addcontentsline=\nocontentsline#1{#2}\egroup}
\newcommand{\eps}{\varepsilon}
\renewcommand{\phi}{\varphi}
\renewcommand{\d}{\partial}
\newcommand{\re}{\mathop{\rm Re} }
\newcommand{\im}{\mathop{\rm Im}}
\newcommand{\mR}{\mathbb{R}}
\newcommand{\mC}{\mathbb{C}}
\newcommand{\mN}{\mathbb{N}} 
\newcommand{\mZ}{\mathbb{Z}} 
\newcommand{\mK}{\mathbb{K}}
\newcommand{\supp}{\mathop{\rm supp}}
\newcommand{\abs}[1]{\lvert #1 \rvert}
\newcommand{\norm}[1]{\lVert #1 \rVert}
\newcommand{\csubset}{\Subset}
\newcommand{\detg}{\lvert g \rvert}
\newcommand{\msetminus}{\setminus}

\newcommand{\mF}{\mathscr{F}}

\newcommand{\br}[1]{\langle #1 \rangle}

\newcommand{\ehat}{\,\hat{\rule{0pt}{6pt}}\,}
\newcommand{\echeck}{\,\check{\rule{0pt}{6pt}}\,}
\newcommand{\etilde}{\,\tilde{\rule{0pt}{6pt}}\,}

\newcommand{\tr}{\mathrm{tr}}
\newcommand{\mdiv}{\mathrm{div}}

\newcommand{\mE}{\mathbb{E}}

\newcommand{\sidenote}[1]{
  \refstepcounter{sidenote}\mbox{\textsuperscript{\alph{sidenote}}}%
  \marginpar{\bf \footnotesize\raggedright\strut\mbox{\textsuperscript{\alph{sidenote}} }#1}%
}
\newcounter{sidenote}
\setlength{\marginparwidth}{.8in}

\title[Strictly convex corners scatter]{Strictly convex corners scatter}

\author{Lassi P\"aiv\"arinta}
\address{Department of Mathematics and Statistics, University of Helsinki}
\email{lassi.paivarinta@helsinki.fi}

\author{Mikko Salo}
\address{Department of Mathematics and Statistics, University of Jyv\"askyl\"a}
\email{mikko.j.salo@jyu.fi}

\author{Esa V. Vesalainen}
\address{Department of Mathematics and Statistics, University of Helsinki}
\email{esa.vesalainen@helsinki.fi}


\begin{abstract}
We prove the absence of non-scattering energies for potentials in the plane having a corner of angle smaller than $\pi$. This extends the earlier result of Bl{\aa}sten, P\"aiv\"arinta and Sylvester who considered rectangular corners. In three dimensions, we prove a similar result for any potential with a circular conic corner whose opening angle is outside a countable subset of $(0,\pi)$.
\end{abstract}

\maketitle


\section{Introduction} \label{sec_intro}

\subsection{Background}

This article is concerned with \emph{non-scattering energies}. These are energies $\lambda > 0$ for which there exists a nontrivial incident wave that does not scatter (equivalently, the far field operator at energy $\lambda > 0$ has nontrivial kernel).

Certain reconstruction methods in inverse scattering theory, such as the linear sampling method \cite{Colton--Kirsch} or the factorization method \cite{Kirsch_factorization}, may fail at non-scattering energies and therefore these energies are to be avoided. This has led people to study the usually larger class of interior transmission eigenvalues which first appeared in \cite{Colton--Monk, Kirsch}. For acoustic scattering, the transmission eigenvalues often form an infinite discrete set \cite{Paivarinta--Sylvester, Cakoni--Gintides--Haddar}, and in recent years they have been studied intensively. For more information about transmission eigenvalues, we recommend the survey \cite{Cakoni--Haddar} as well as the articles mentioned in the recent editorial \cite{Cakoni--Haddar2} and their references.

Results on non-scattering energies appear to be scarce, apart from discreteness results which follow from corresponding results for transmission eigenvalues. For radial compactly supported potentials, the set of non-scattering energies coincides with the infinite discrete set of transmission eigenvalues \cite{Colton--Monk}. Non-scattering energies are exactly those energies for which the scattering matrix has $1$ as an eigenvalue, and \cite{GH} constructs $C^{\infty}_c$ potentials with no non-scattering energies.

It is relevant to mention here the related topic of transparent potentials. These are nonzero potentials whose far field operator is identically zero at some fixed energy $\lambda$, and thus no incident wave with energy $\lambda$ scatters. 
Several constructions of such radial potentials have been given, see e.g.\ the works \cite{Regge, newton, sabatier, Grinevich--Manakov, GrinNov}.

The recent work \cite{BPS} suggests that corner points in the scattering potential always generate a scattered wave. More precisely, \cite{BPS} proves the absence of non-scattering energies in acoustic scattering for certain contrasts having corners with a right angle, i.e.~for contrasts supported in a rectangle $K$ that do not vanish at a corner point of $K$. In this paper we extend this result to corners with arbitrary opening angle $< \pi$ in two dimensions. In three dimensions, we prove that circular conic corner points with angle outside an at most countable subset of $(0,\pi)$ lead to absence of non-scattering energies. By the ``opening angle'' of a cone
\[\bigl\{ (x',x_n) \in \mR^{n-1}\times\mathbb R \bigm| \abs{x'}\leq cx_n \bigr\},\]
say, where $c\in\mathbb R_+$, we mean the angle $\vartheta\in\left(0,\pi\right)$ such that $\tan(\vartheta/2)=c$.

\subsection{Non-scattering energies}
Let us state the precise definition of non-scattering energies. This notion makes sense in the context of short range scattering theory in $\mR^n$, and we will formulate our results using the notation of quantum mechanical scattering following \cite[Chapter XIV]{H2}. We will discuss later the analogous case of acoustic scattering, where the term \emph{non-scattering wavenumbers} is more appropriate.

Let $V$ be a short range potential, which in this paper will mean that $V$ is in $L^{\infty}(\mR^n)$ and there are $C > 0$, $\eps > 0$ such that 
$$
\abs{V(x)} \leq C \br{x}^{-1-\eps} \quad \text{a.e.~in $\mR^n$.}
$$
Here we write $\br{x} = (1+\abs{x}^2)^{1/2}$. 
For any $\lambda > 0$ and $g \in L^2(S^{n-1})$, we consider the \emph{incident wave} 
\begin{equation} \label{intro_u0_def}
u_0(x) = \int_{S^{n-1}} e^{i \sqrt{\lambda} x \cdot \omega} g(\omega) \,d\omega,
\end{equation}
which solves the free Schr\"odinger equation $(-\Delta-\lambda) u_0 = 0$ in $\mR^n$. Corresponding to the incident wave $u_0$, there is a unique solution of the perturbed Schr\"odinger equation 
$$
(-\Delta + V - \lambda) u = 0 \quad \text{in $\mR^n$}
$$
having the form 
$$
u = u_0 + v,
$$
where $v$ satisfies the \emph{outgoing radiation condition}. There are many equivalent formulations of this condition that selects the unique outgoing solution of the Schr\"odinger equation: we follow \cite{H2} and say that $v$ satisfies the outgoing radiation condition if $v = (-\Delta-\lambda-i0)^{-1} f$ for some $f$ in the Agmon--H\"ormander space $B$ (see Section \ref{sec_orthogonality_identity} for the precise definitions). The function $v$ is called the \emph{outgoing scattered wave} corresponding to $u_0$, and $u$ is called the \emph{total wave}.

If $u_0$ corresponds to $g \in L^2(S^{n-1})$ as above and if $x = r\theta$ where $\theta \in S^{n-1}$, then $u_0$ has the following asymptotics as $r \to \infty$:
$$
u_0(r\theta) \sim c r^{-\frac{n-1}{2}} (e^{i\sqrt{\lambda} r} g(\theta) + i^{n-1} e^{-i\sqrt{\lambda} r} g(-\theta)).
$$
The scattered wave $v$ has the asymptotics 
$$
v(r\theta) \sim c r^{-\frac{n-1}{2}} e^{i\sqrt{\lambda} r} A_{\lambda} g(\theta)
$$
where $A_{\lambda}$ is the \emph{far field operator}, which is a bounded linear operator  
$$
A_{\lambda}: L^2(S^{n-1}) \to L^2(S^{n-1}).
$$
The function $A_{\lambda} g$ is called the \emph{far field pattern} of the scattered wave $v$. If the far field pattern vanishes and additionally $V$ is compactly supported, the Rellich uniqueness theorem (see \cite{Vesalainen_rellich} for references) implies that also the scattered wave $v$ must be compactly supported. Thus the vanishing of the far field pattern may be interpreted so that the incident wave $u_0$ does not produce any scattered wave at infinity.

We may then divide all energies $\lambda > 0$ in two classes: those for which all nontrivial incident waves scatter, and those for which there exist nontrivial incident waves that cannot be observed at infinity. The latter case is the case of non-scattering energies:

\begin{definition}
Let $V$ be a short range potential in $\mR^n$. We say that $\lambda > 0$ is a non-scattering energy for the potential $V$, if there exists a nonzero $g \in L^2(S^{n-1})$ for which $A_{\lambda} g = 0$.
\end{definition}

\subsection{Main results}
Our argument for the absence of non-scattering energies is based on suitable complex geometrical optics solutions to the Schr\"odinger equation. Since these solutions grow exponentially at infinity, it will be natural to assume that the potentials satisfy a corresponding decay condition.

\begin{definition}
$V \in L^{\infty}(\mR^n)$ is called \emph{superexponentially decaying} if for any $\gamma > 0$ there is $C_{\gamma} > 0$ such that $\abs{V(x)} \leq C_{\gamma} e^{-\gamma \abs{x}}$ a.e.~in $\mR^n$.
\end{definition}

The main results of this paper are as follows. Below we will write $\chi_E$ for the characteristic function of a set $E$ and $C^s(\mR^n)$ for the H\"older spaces with norm (for $0 < s < 1$) 
$$
\norm{f}_{C^s} = \norm{f}_{L^{\infty}} + \sup_{x \neq y} \frac{\abs{f(x)-f(y)}}{\abs{x-y}^s}.
$$

\begin{thm} \label{thm_mainthm1}
Let $V(x) = \chi_C(x) \varphi(x)$ where $C \subset \mR^2$ is a closed strictly convex cone with its vertex at the origin, and $\varphi$ is a superexponentially decaying function in $\mR^2$ such that $\br{x}^{\alpha} \varphi \in C^s(\mR^2)$ for some $\alpha > 5/3$ and $s > 0$. Let also $\varphi(0) \neq 0$. Then there are no non-scattering energies for the potential $V$.
\end{thm}

\begin{thm} \label{thm_mainthm2}
Let $V(x) = \chi_C(x) \varphi(x)$ where $C \subset \mR^3$ is a closed circular cone with opening angle $\gamma\in(0,\pi)$ having its vertex at the origin, and $\varphi$ is a superexponentially decaying function in $\mR^3$ such that $\br{x}^{\alpha} \varphi \in C^s(\mR^3)$ for some $\alpha > 9/4$ and $s > 1/4$. Let also $\varphi(0) \neq 0$.

There exists an at most countable subset $E \subset (0,\pi)$ such that if $V$ is as above and if $\gamma \in (0,\pi) \setminus E$, then there are no non-scattering energies for the potential $V$.
\end{thm}

\begin{remark}
The technical conditions for $\varphi$ in the above theorems are satisfied for instance if $\varphi$ is a compactly supported $s$-H\"older continuous function in $\mR^n$ where $s > 0$ if $n=2$ or $s > 1/4$ if $n=3$.
\end{remark}

The above theorems also imply analogous statements in acoustic scattering. In this case we consider $\mR^n$ as a medium where acoustic waves propagate, and the refractive index of the medium is assumed to be $(1+m)^{1/2}$ where the contrast $m$ satisfies the short range condition (one often writes $n^2=1+m$  where $n$ is the refractive index). Let $k > 0$ be a wavenumber, and write $\lambda=k^2$. We consider an incident wave $u_0$ as in \eqref{intro_u0_def} solving $(-\Delta-k^2) u_0 = 0$ in $\mR^n$. There is a corresponding total wave $u = u_0 + v$ that solves  
$$
(-\Delta-k^2(1+m)) u = 0 \quad \text{in } \mR^n,
$$
where the scattered wave $v$ satisfies the outgoing radiation condition.

Now, we say that $k>0$ is a non-scattering wavenumber for the contrast $m$ if there is a nontrivial incident wave $u_0$ such that the scattered wave $v$ has trivial asymptotics at infinity. The proofs of Theorems \ref{thm_mainthm1} and \ref{thm_mainthm2} apply in this situation, and we obtain corresponding results which state the absence of non-scattering wavenumbers for contrasts $m$ that satisfy exactly the same conditions as the potentials $V$.

We mention that the method of complex geometrical optics solutions that is used for proving the above theorems has a long history both in inverse scattering problems \cite{novikovkhenkin, nachman, novikov, ramm} and inverse boundary value problems \cite{calderon, SU}. See \cite{Novikov_survey, U_IP} for further references.

\subsection{Structure of argument}

We follow the approach of \cite{BPS} and argue by contradiction. Assume that $\lambda\in\mathbb R_+$ is a non-scattering energy for the potential $V$. Then, we have nontrivial solutions $w,v_0\in B_2^\ast$ to
\[(-\Delta+V-\lambda)\,w=0,\quad\text{and}\quad(-\Delta-\lambda)\,v_0=0\]
in $\mathbb R^n$, satisfying $w-v_0\in\mathring B_2^\ast$. Since $v_0$ is real analytic, the lowest degree non-vanishing terms in its Taylor series at the origin form a harmonic homogeneous polynomial $H(x)\not\equiv0$ of degree $N \geq 0$. The desired contradiction will be obtained by showing that $H(x)\equiv0$.

It is proved in Section \ref{sec_orthogonality_identity} that for non-scattering energies $\lambda$,
\[\int_{\mathbb R^n}Vuv_0=0\]
for any $u\in e^{\gamma\left\langle\cdot\right\rangle}L^2(\mathbb R^n)$ solving $(-\Delta+V-\lambda)\,u=0$, where $\gamma\in\mathbb R_+$ is arbitrary.

In Section \ref{CGO-section}, we shall discuss the existence of solutions of the form $u=e^{-\rho\cdot x}\,(1+\psi)$ to $(-\Delta+V-\lambda)\,u=0$ for $\rho\in\mathbb C^n$ with $\rho\cdot\rho=-\lambda$ and $\psi$ being well controlled as $\left|\rho\right|\longrightarrow\infty$. In Section \ref{sec_reduction_laplace} we show that substituting the complex geometrical optics to the above integral identity implies the vanishing of a certain Laplace transform. More precisely, after some estimations we see that
\[\int_Ce^{-\rho\cdot x}\,H(x)\, dx\lesssim\left|\rho\right|^{-N-2-\beta},\]
for some small $\beta\in\mathbb R_+$,
as $\left|\rho\right|\longrightarrow\infty$, and we restrict to a suitable subset of vectors $\rho\in\mathbb C^n$ with $\rho\cdot\rho=-\lambda$. On the other hand, from the homogeneity of $H(x)$, we see that
\[\int_Ce^{-\rho\cdot x}\,H(x)\, dx=\left|\rho\right|^{-N-2}\int_Ce^{-\rho/\left|\rho\right|\cdot x}\,H(x)\, dx,\]
for the same $\rho$ as before. The last two estimates turn out to be compatible only if
\[\int_Ce^{-\rho\cdot x}\,H(x)\, dx=0\]
for certain vectors $\rho\in\mathbb C^n$ with $\rho\cdot\rho=0$ (as opposed to $\rho\cdot\rho=-\lambda$). The last identity asserts the vanishing of the Laplace transform of $\chi_C H$, where $H$ is a harmonic homogeneous polynomial and $\chi_C$ is the characteristic function of the cone $C$, for certain complex vectors.

Up to this point, we have closely followed the approach of \cite{BPS}. We now depart from this approach and move to polar coordinates in the Laplace transform identity. This implies the vanishing of the following integrals over a spherical cap for certain $\rho \in \mC^n$: 
\[\int_{C\cap S^{n-1}} (\rho \cdot x)^{-N-n} \,H(x)\, dS(x)=0.\]
As a restriction of a harmonic homogeneous polynomial, $H$ can be expanded in terms of spherical harmonics of fixed degree. The main difference between our approach and \cite{BPS} is that we will perform computations in terms of these spherical harmonics.

In Section \ref{sec_two_dimensions} we discuss the case $n=2$, which is particularly simple. There $H(x)$ must be of the form $a(x+iy)^N+b(x-iy)^N$ for some constants $a,b\in\mathbb C$. When this is inserted in the above vanishing relation, the ensuing integrals can be evaluated explicitly and one obtains a concrete homogeneous linear pair of equations for $a$ and $b$. It is not difficult to prove that the corresponding determinant is nonzero, and so we conclude that $a=b=0$ and $H(x)\equiv0$ as desired.

Section \ref{sec_three_dimensions} considers the three-dimensional case. The polynomial $H(x)$ can be written as a finite linear combination of spherical harmonics, and one can again obtain a homogeneous linear system; this time the ``unknowns'' are the constant coefficients multiplied by certain concrete but complicated integrals, and the determinant of the system can be arranged to be a Vandermonde determinant. Thus, the vanishing of the coefficients of $H(x)$ is reduced to proving that all of these complicated integrals are nonzero. It is not clear to us how to do so. However, the integrals depend analytically on the opening angle, and we can prove that they are not identically zero as functions of the opening angle. In this way we get the desired contradiction when the opening angle is outside some at most countable set of exceptional angles.

We remark that there are two complications in extending the methods to dimensions $n \geq 4$. First of all, the construction of complex geometrical optics solutions is carried out by a Neumann series argument where the potential $V$ appears as a pointwise multiplier. The fact that $V$ is not very regular (it is essentially the characteristic function of a cone) implies that our construction of complex geometrical optics solutions, which is based on $L^p$ estimates from \cite{KRS} that were also used in an early version of \cite{BPS}, only gives good enough estimates when $n=2,3$. In \cite{BPS} another construction of solutions was employed, this construction works when $n \geq 2$ but seems to apply to ``polygonal'' cones instead of the circular cones that we use.  The second complication is related to the more complex structure of spherical harmonics in high dimensions, which makes the resulting integrals difficult to evaluate.

\subsection*{Acknowledgements}

All three authors were partly supported by the Academy of Finland (Finnish Centre of Excellence in Inverse Problems Research). In addition, L.P.\ was supported by an ERC Advanced Grant, M.S.\ was supported by an ERC Starting Grant, and E.V.V.\ was supported by Finland's Ministry of Education through the Doctoral Program of Inverse Problems, by the Vilho, Yrj\"o and Kalle V\"ais\"al\"a Foundation, and by the aforementioned ERC Advanced Grant. E.V.V.\ is grateful to Eemeli Bl{\aa}sten for many enlightening discussions on corner scattering and related topics.

\section{Short range scattering} \label{sec_orthogonality_identity}

In this section we recall some basic facts in short range scattering theory that are required for the setup in this paper. The following proposition is the main result in this section, and only its statement will be used in the subsequent sections.

\begin{prop} \label{prop:exponential_orthogonality}
Let $V$ be a superexponentially decaying potential. If $\lambda > 0$ is a non-scattering energy for $V$, then 
$$
\int_{\mR^n} V u v_0 \,dx = 0
$$
for some nontrivial solution $v_0 \in B^*$ of $(-\Delta-\lambda)v_0 = 0$ in $\mR^n$, and for all $u \in L^2_{loc}(\mR^n)$ such that $(-\Delta+V-\lambda)u=0$ in $\mR^n$ and $u \in e^{\gamma \br{x}} L^2(\mR^n)$ for some $\gamma > 0$.
\end{prop}

The results in this section are stated in terms of Agmon--H\"ormander spaces $B$ and $B^*$. The basic reference is \cite[Chapter XIV]{H2}. However, most of the next results are also contained in \cite{PSU_magnetic} (see also \cite{melrose}, \cite{uhlmannvasy}) in a convenient form. Thus the reader may refer to \cite{PSU_magnetic} for proofs and further details on the statements in this section.

\subsection{Function spaces}

The space $B$ (see \cite[Section 14.1]{H2}) is the set of those $u \in L^2(\mR^n)$ for which the norm 
\begin{equation*}
\norm{u}_B = \sum_{j=1}^{\infty} ( 2^{j-1} \int_{X_j} \abs{u}^2 \,dx )^{1/2}
\end{equation*}
is finite. Here $X_1 = \{ \abs{x} < 1 \}$ and $X_j = \{ 2^{j-2} < \abs{x} < 2^{j-1} \}$ for $j \geq 2$. This is a Banach space whose dual $B^*$ consists of all $u \in L^2_{\text{loc}}(\mR^n)$ such that 
\begin{equation*}
\norm{u}_{B^*} = \sup_{R > 1} \left[ \frac{1}{R} \int_{\abs{x} < R} \abs{u}^2 \,dx \right]^{1/2} < \infty.
\end{equation*}
The set $C^{\infty}_c(\mR^n)$ is dense in $B$ but not in $B^*$. The closure in $B^*$ is denoted by $\mathring{B}^*$, and $u \in B^*$ belongs to $\mathring{B}^*$ if and only if 
\begin{equation*}
\lim_{R \to \infty} \frac{1}{R} \int_{\abs{x} < R} \abs{u}^2 \,dx = 0.
\end{equation*}
We will also need the Sobolev space variant $B^*_2$ of $B^*$, defined via the norm 
\begin{equation*}
\norm{u}_{B^*_2} = \sum_{\abs{\alpha} \leq 2} \norm{D^{\alpha} u}_{B^*}.
\end{equation*}


If $\lambda > 0$ we will consider the sphere $M_{\lambda} = \{ \xi \in \mR^n \,;\, \abs{\xi} = \sqrt{\lambda} \}$ with Euclidean surface measure $dS_{\lambda}$. The corresponding $L^2$ space is $L^2(M_{\lambda}) = L^2(M_{\lambda}, dS_{\lambda})$, and of course $L^2(S^{n-1}) = L^2(M_1)$.

\subsection{Scattering solutions}

We consider scattering in $\mR^n$ with respect to incident waves $u_0$ with fixed energy $\lambda > 0$, where $u_0$ has the form 
$$
u_0 = P_0(\lambda) g, \qquad g \in L^2(M_{\lambda}).
$$
Here $P_0(\lambda)$ is the \emph{Poisson operator} 
$$
P_0(\lambda) g(x) = \frac{i}{(2\pi)^{n-1}} \int_{M_{\lambda}} e^{i x \cdot \xi} g(\xi) \,\frac{dS_{\lambda}(\xi)}{2\sqrt{\lambda}}, \qquad x \in \mR^n.
$$
Thus $u_0$ is a Herglotz wave corresponding to a pattern $g$ at infinity. The function $u_0$ belongs to $B_2^*$ and it satisfies the Helmholtz equation 
$$
(-\Delta-\lambda)u_0 = 0 \qquad \text{in } \mR^n.
$$

Now consider quantum scattering in $\mR^n$ where the medium properties are described by a short range potential $V$. By this we mean that $V \in L^{\infty}(\mR^n)$ is real valued, and for some $C > 0, \eps > 0$ one has 
$$
\abs{V(x)} \leq C \br{x}^{-1-\eps} \quad \text{for a.e.~$x \in \mR^n$}.
$$
The outgoing resolvent $R_V(\lambda + i0) = (-\Delta+V-\lambda-i0)^{-1}$ is well defined for all $\lambda > 0$, and it is a bounded operator 
$$
R_V(\lambda + i0): B \to B^*_2.
$$

For any incoming wave $u_0 = P_0(\lambda) g$ where $g \in L^2(M_{\lambda})$, there is a unique total wave $u$ solving the equation 
$$
(-\Delta + V - \lambda) u = 0 \qquad \text{in } \mR^n
$$
such that $u = u_0 + v$ where $v$ is \emph{outgoing} (meaning that $v = R_0(\lambda+i0)f$ for some $f \in B$). In fact, if $u_0 = P_0(\lambda) g$, then one has 
$$
u = P_V(\lambda)g
$$
where $P_V(\lambda): L^2(M_{\lambda}) \to B^*_2(\mR^n)$ is the outgoing Poisson operator 
$$
P_V(\lambda)g = P_0(\lambda)g - R_V(\lambda+i0)(V P_0(\lambda) g).
$$

\subsection{Asymptotics}

We write $u \sim u_0$ to denote that $u-u_0 \in \mathring{B}^*$, which is interpreted so that $u$ and $u_0$ have the same asymptotics at infinity. Now if $g \in L^2(M_{\lambda})$, then $P_V(\lambda)g$ has asymptotics 
\begin{equation*}
P_V(\lambda) g \sim c_{\lambda} r^{-\frac{n-1}{2}} \left[ e^{i\sqrt{\lambda} r} (S_V(\lambda) g)(\sqrt{\lambda} \theta) + i^{n-1} e^{-i\sqrt{\lambda} r} g(-\sqrt{\lambda}\theta) \right]
\end{equation*}
as $r = \abs{x} \to \infty$, where $x = r \theta$ and $c_{\lambda} = (\sqrt{\lambda}/2\pi i)^{\frac{n-3}{2}} / 4\pi$. Here 
$$
S_V(\lambda): L^2(M_{\lambda}) \to L^2(M_{\lambda})
$$
is the \emph{scattering matrix} for $V$ at energy $\lambda$. It is a unitary operator, $S_V(\lambda)^* S_V(\lambda) = \text{Id}$, and if $V = 0$ one has $S_0(\lambda) = \text{Id}$. The operator 
$$
A_V(\lambda)=  S_V(\lambda) - S_0(\lambda): L^2(M_{\lambda}) \to L^2(M_{\lambda})
$$
is called the \emph{far field operator}, and $A_V(\lambda)g$ is the \emph{far field pattern} of the outgoing scattered wave $v$ at infinity.

Recall from the introduction that if $\lambda > 0$ is such that there exists a nontrivial Herglotz wave $v_0 = P_0(\lambda) g$ for which the far field pattern $A_V(\lambda)g$ is identically zero, we say that $\lambda$ is a non-scattering energy. Thus, $\lambda$ is a non-scattering energy if and only if there is a nontrivial function $g \in L^2(M_{\lambda})$ such that 
$$
(S_V(\lambda) - S_0(\lambda))g = 0.
$$

\subsection{Orthogonality identities}

We now recall the ''boundary pairing'' for scattering solutions \cite{melrose}, \cite[Proposition 2.3]{PSU_magnetic}.

\begin{prop} \label{prop:boundary_pairing}
Let $u, v \in B^*$ and $(H_0-\lambda)u \in B, (H_0-\lambda)v \in B$. If $u$ and $v$ have the asymptotics 
\begin{align*}
u &\sim r^{-\frac{n-1}{2}} \left[ e^{i\sqrt{\lambda} r} g_+(\sqrt{\lambda} \theta) + e^{-i\sqrt{\lambda} r} g_-(-\sqrt{\lambda} \theta) \right], \\
v &\sim r^{-\frac{n-1}{2}} \left[ e^{i\sqrt{\lambda} r} h_+(\sqrt{\lambda} \theta) + e^{-i\sqrt{\lambda} r} h_-(-\sqrt{\lambda} \theta) \right]
\end{align*}
for some $g_{\pm}, h_{\pm} \in L^2(M_{\lambda})$, then 
\begin{multline*}
(u|(H_0-\lambda)v)_{\mR^n} - ((H_0-\lambda)u|v)_{\mR^n} \\
 = 2i \lambda^{-\frac{n-2}{2}} \left[ (g_+|h_+)_{M_{\lambda}} - (g_-|h_-)_{M_{\lambda}} \right].
\end{multline*}
Here $(u|v)_{\mR^n} = \int_{\mR^n} u \bar{v} \,dx$ and $(g|h)_{M_{\lambda}} = \int_{M_{\lambda}} g \bar{h} \,dS_{\lambda}$.
\end{prop}

As a consequence, the existence of a nontrivial $g \in L^2(M_{\lambda})$ for which $A_V(\lambda)g = 0$ is characterized by the following orthogonality identity:

\begin{prop} \label{prop:boundary_pairing_application}
Let $V$ be a short range potential, let $\lambda > 0$, and let $g \in L^2(M_{\lambda})$. Then $A_V(\lambda) g = 0$ if and only if for all $f \in L^2(M_{\lambda})$ one has 
$$
\int_{\mR^n} V u v_0 \,dx = 0
$$
where $u = P_V(\lambda) f$ and $v_0 = P_0(\lambda) g$.
\end{prop}
\begin{proof}
Apply Proposition \ref{prop:boundary_pairing} with $u = P_V(\lambda)S_V(\lambda)^* f$ and $v = v_0$ to obtain 
\begin{align*}
\int_{\mR^n} V u v_0 \,dx &= 2i \lambda^{-\frac{n-2}{2}} c_{\lambda}^2 \Big[ (S_V(\lambda) S_V(\lambda)^* f | g)_{M_{\lambda}} - (S_V(\lambda)^* f| g)_{M_{\lambda}} \Big] \\
 &= -2i \lambda^{-\frac{n-2}{2}} c_{\lambda}^2 (f| (S_V(\lambda) - S_0(\lambda))g)_{M_{\lambda}}
\end{align*}
since $S_V(\lambda)$ is unitary and $S_0(\lambda) = \text{Id}$. Now $A_V(\lambda)g = 0$ if and only if the integral over $\mR^n$ vanishes for all $f \in L^2(M_{\lambda})$.
\end{proof}

The last proposition shows that if $A_V(\lambda)g = 0$, then the function $V v_0$ where $v_0 = P_0(\lambda)g$ is orthogonal to all scattering solutions $P_V(\lambda) f$ where $f \in L^2(M_{\lambda})$. It is well known that if $V$ has certain decay properties, then solutions of $(-\Delta+V-\lambda)u = 0$ satisfying corresponding growth conditions can be approximated by scattering solutions. The next result from \cite{uhlmannvasy} (see also \cite[Proposition 2.4]{PSU_magnetic}) concerns exponentially decaying potentials.

\begin{prop} \label{prop:density}
Assume that $V \in L^{\infty}(\mR^n)$ satisfies $\abs{V(x)} \leq C e^{-\gamma_0 \br{x}}$ for some $\gamma_0 > 0$. Let $0 < \gamma < \gamma_0$. Given any $u \in e^{\gamma \br{x}} L^2$ with $(-\Delta+V-\lambda)u = 0$, there exist $g_j \in L^2(M_{\lambda})$ such that $P_V(\lambda) g_j \to u$ in $e^{\gamma_0 \br{x}} L^2$.
\end{prop}
 
We will later want to use complex geometrical optics solutions $u$. Since these may have arbitrarily large exponential growth, it is natural to assume that the potential is superexponentially decaying.

\begin{prop} \label{prop:exponential_orthogonality2}
Let $V$ be a superexponentially decaying potential, let $\lambda > 0$, and let $g \in L^2(M_{\lambda})$. Then $A_V(\lambda) g = 0$ if and only if one has 
$$
\int_{\mR^n} V u v_0 \,dx = 0
$$
for $v_0 = P_0(\lambda) g$ and for all $u \in L^2_{loc}(\mR^n)$ such that $(-\Delta+V-\lambda)u=0$ in $\mR^n$ and $u \in e^{\gamma \br{x}} L^2(\mR^n)$ for some $\gamma > 0$.
\end{prop}
\begin{proof}
Follows by combining Propositions \ref{prop:boundary_pairing_application} and \ref{prop:density}.
\end{proof}

\begin{proof}[Proof of Proposition \ref{prop:exponential_orthogonality}]
Follows immediately from Proposition \ref{prop:exponential_orthogonality2}.
\end{proof}

\section{Complex geometrical optics solutions}\label{CGO-section}

By Proposition \ref{prop:exponential_orthogonality} we know that if $V$ is superexponentially decaying and if $\lambda > 0$ is a non-scattering energy for $V$, then there exists a nontrivial $v_0 \in B^*$ satisfying $(-\Delta-\lambda)v_0 = 0$ such that we have 
$$
\int_{\mR^n} V u v_0 \,dx = 0
$$
for any exponentially growing solution $u$ of $(-\Delta+V-\lambda)u = 0$.

We will employ complex geometrical optics solutions as the solutions $u$ above. It will be important to have $L^q$ estimates for large $q$ with suitable decay for the remainder term $\psi$ in these solutions. In \cite{BPS}, such solutions were constructed in all dimensions $n \geq 2$ but for ``polygonal'' cones $C$ (that is, the cross-section of the cone is a polygon). For our higher dimensional result it is more convenient to use circular cones, and it turns out that the argument of \cite{BPS} is not valid in this case. Therefore we base our construction on certain $L^p$ estimates from \cite{KRS}. This argument gives sufficient estimates for $n=2, 3$ but not for $n \geq 4$.

\begin{thm} \label{thm_cgo}
Let $\lambda > 0$, assume that $n \in \{2,3\}$, and let $q=\frac{2(n+1)}{n-1}$. Let $V(x) = \chi_{C}(x) \varphi(x)$ where $C \subset \mR^n$ is a closed circular cone with opening angle $< \pi$, and where $\varphi$ satisfies $\br{x}^{\alpha} \varphi \in C^s(\mR^n)$ for some $\alpha > 5/3$ and $s > 0$ if $n=2$ (resp.\ $\alpha > 9/4$ and $s > 1/4$ if $n=3$).

If $\rho \in \mC^n$ satisfies $\rho \cdot \rho = -\lambda$ and $\abs{\im(\rho)}$ is sufficiently large, there is a solution of 
$$
(-\Delta+V-\lambda)u = 0 \quad \text{in $\mR^n$}
$$
of the form $u = e^{-\rho \cdot x}(1+\psi)$, where $\psi$ satisfies for some $\delta > 0$ 
$$
\norm{\psi}_{L^q(\mR^n)} = O(\abs{\im(\rho)}^{-n/q-\delta}) \text{ as $\abs{\rho} \to \infty$}.
$$
\end{thm}

We begin by stating some a priori inequalities. Below we will write $D = -i\nabla$, $\mathscr{S}$ will be the space of Schwartz test functions, and $H^{s,p}$ for $s \in \mR$ and $1<p<\infty$ will be the Bessel potential space with norm $\norm{u}_{H^{s,p}} = \norm{\br{D}^s f}_{L^p}$ with $\br{D} = (1+D^2)^{1/2}$. Also, $r'$ will denote the H\"older conjugate exponent of $r$ (so that $1/r + 1/r' = 1$).

\begin{prop} \label{prop_krs}
Let $n \geq 2$, let $1 < r < 2$, and assume that 
$$
\frac{1}{r} - \frac{1}{r'} \in \left\{ \begin{array}{cc} {\left[\frac{2}{n+1}, \frac{2}{n}\right]} & \text{if }n \geq 3, \\[5pt] {\left[\frac{2}{n+1}, \frac{2}{n}\right)} & \text{if } n=2. \end{array} \right.
$$
There is a constant $M > 0$ such that for any $\zeta \in \mC^n$ with $\re(\zeta) \neq 0$, one has 
$$
\norm{f}_{L^{r'}} \leq M \abs{\re(\zeta)}^{n(1/r-1/r')-2} \norm{(-\Delta+2\zeta \cdot D)f}_{L^r}, \quad f \in \mathscr{S}.
$$
\end{prop}
\begin{proof}
This is a consequence of the uniform Sobolev inequalities in \cite{KRS}. In particular, the case $n \geq 3$ follows from \cite[Theorem 2.4]{KRS} after dilations and conjugations by the exponentials $e^{\pm i \re(\zeta) \cdot x}$. For the case $n=2$ and more details, see \cite[Section 5.3]{Ruiz}.
\end{proof}

The next result shows solvability for an inhomogeneous equation related to complex geometrical optics solutions.

\begin{prop} \label{prop_cgo1}
Let $n \geq 2$, let $1 < r < 2$, and assume that 
$$
\frac{1}{r} - \frac{1}{r'} \in \left[\frac{2}{n+1}, \frac{2}{n}\right).
$$
Let $s \in \mR$, and let $V$ be a measurable function in $\mR^n$ satisfying the following multiplier property for some constant $A > 0$:
$$
\norm{Vf}_{H^{s,r}} \leq A \norm{f}_{H^{s,r'}}, \qquad f \in \mathscr{S}.
$$
There exist constants $C > 0$ and $R > 0$ such that whenever $\zeta \in \mC^n$ satisfies $\abs{\re(\zeta)} \geq R$, then for any $f \in H^{s,r}(\mR^n)$ the equation 
$$
(-\Delta + 2 \zeta \cdot D + V) u = f \quad \text{in } \mR^n
$$
has a solution $u \in H^{s,r'}(\mR^n)$ satisfying 
$$
\norm{u}_{H^{s,r'}} \leq C \abs{\re(\zeta)}^{n(1/r-1/r')-2}\norm{f}_{H^{s,r}}.
$$
\end{prop}
\begin{proof}
Let $\zeta \in \mC^n$ with $\re(\zeta) \neq 0$. We will use a standard duality argument to obtain a solvability result from the a priori estimates in Proposition \ref{prop_krs}. Write $P = (-\Delta + 2\zeta \cdot D)$, so the formal adjoint is $P^* =  (-\Delta + 2 \overline{\zeta} \cdot D)$. Applying Proposition \ref{prop_krs} to $\br{D}^{-s} w$, we have the inequality 
\begin{equation} \label{w_apriori}
\norm{w}_{H^{-s,r'}} \leq M \abs{\re(\zeta)}^{n(1/r-1/r')-2} \norm{P^* w}_{H^{-s,r}}, \quad w \in \mathscr{S}.
\end{equation}
Fix $f \in H^{s,r}$ and define a linear functional 
$$
l: P^*(\mathscr{S}) \subset H^{-s,r} \to \mC,\ \ l(P^* w) = (w,f)
$$
where $(\,\cdot\,,\,\cdot\,)$ is the distributional pairing, defined to be conjugate linear in the second argument. By \eqref{w_apriori} any element of $P^*(\mathscr{S})$ has a unique representation as $P^* w$ for some $w \in \mathscr{S}$, so $l$ is well defined and satisfies 
\begin{align*}
\abs{l(P^* w)} &= \abs{(w,f)} \leq \norm{w}_{H^{-s,r'}} \norm{f}_{H^{s,r}} \\
 &\leq M \abs{\re(\zeta)}^{n(1/r-1/r')-2} \norm{f}_{H^{s,r}} \norm{P^* w}_{H^{-s,r}}.
\end{align*}
By Hahn-Banach we may extend $l$ as a continuous linear functional $\bar{l}: H^{-s,r} \to \mC$ with the same norm bound, and by duality there is $v \in H^{s,r'}$ such that $\bar{l}(w) = (w,v)$ for $w \in H^{-s,r}$ and 
$$
\norm{v}_{H^{s,r'}} \leq M \abs{\re(\zeta)}^{n(1/r-1/r')-2} \norm{f}_{H^{s,r}}.
$$
If $w \in \mathscr{S}$ we have 
$$
(w,Pv) = (P^* w, v) = \bar{l}(P^* w) = l(P^* w) = (w,f)
$$
so that $Pv = f$.

By the above argument, for any $s \in \mR$ there is a linear operator 
$$
G_{\zeta}: H^{s,r} \to H^{s,r'}, \ \ f \mapsto v
$$
where $v$ solves $(-\Delta+2\zeta \cdot D)v = f$, and one has the norm estimate 
$$
\norm{G_{\zeta} f}_{H^{s,r'}} \leq M \abs{\re(\zeta)}^{n(1/r-1/r')-2} \norm{f}_{H^{s,r}}.
$$
This proves the result in the case $V = 0$.

Let us now consider nonzero $V$. Notice that one has 
$$
\norm{V G_{\zeta} f}_{H^{s,r}} \leq A M \abs{\re(\zeta)}^{n(1/r-1/r')-2} \norm{f}_{H^{s,r}}.
$$
Since $n(1/r-1/r')-2 < 0$, we may choose $R > 0$ so that it satisfies $A M R^{n(1/r-1/r')-2} = 1/2$. Assuming that $\abs{\re(\zeta)} \geq R$, we have 
$$
\norm{V G_{\zeta} f}_{H^{s,r}} \leq \frac{1}{2} \norm{f}_{H^{s,r}}.
$$
Now, we can solve $(-\Delta+2\zeta \cdot D + V)u = f$ by taking $u = G_{\zeta} v$ where $v$ is a solution of  
$$
(\mathrm{Id} + V G_{\zeta})v = f.
$$
By the above estimate, this equation for $v$ can be solved by Neumann series and one has $\norm{v}_{H^{s,r}} \leq 2 \norm{f}_{H^{s,r}}$. Then $u = G_{\zeta} v$ is the required solution and it satisfies 
\begin{align*}
\norm{u}_{H^{s,r'}} &\leq M \abs{\re(\zeta)}^{n(1/r-1/r')-2} \norm{v}_{H^{s,r}} \\
 &\leq 2 M \abs{\re(\zeta)}^{n(1/r-1/r')-2} \norm{f}_{H^{s,r}}. \qedhere
\end{align*}
\end{proof}

\begin{prop} \label{prop_cone_pointwisemultiplier}
Let $n\in\left\{2,3\right\}$, let $r = \frac{2(n+1)}{n+3}$ so that $r'=\frac{2(n+1)}{n-1}$, and assume that $V(x) = \chi_{C}(x) \varphi(x)$ where $C \subset \mR^n$ is a closed circular cone with opening angle $< \pi$ and $\br{x}^{\alpha} \varphi \in C^s(\mR^n)$ for some $\alpha > 5/3$ and $0 < s < 1/2$ if $n=2$ (respectively $\alpha>9/4$ and $1/4 < s < 1/2$ if $n=3$). Then $V \in H^{s-\eps,r}(\mR^n)$ for any $\eps > 0$, and there is a constant $A > 0$ such that 
$$
\norm{Vf}_{H^{s-\eps,r}} \leq A \norm{f}_{H^{s-\eps,r'}}.
$$
\end{prop}
\begin{proof}
We write $V = \tilde{\varphi} g$ where 
\begin{align*}
\tilde{\varphi}(x) &= \br{x}^{\alpha} \varphi(x), \\
g(x) &= \br{x}^{-\alpha} \chi_C(x).
\end{align*}
Also write $q = r'$ and $\tilde{q} = \frac{q}{q-2}$, so that 
$$
(q,\ q',\ \tilde{q}) = \left\{ \begin{array}{ll} (6,\ 6/5,\ 3/2) & \text{if }n=2, \\ (4,\ 4/3,\ 2) & \text{if }n=3. \end{array} \right.
$$
We first observe that, by assumption, 
$$
\tilde{\varphi} \in C^s(\mR^n).
$$
Next, note that Proposition \ref{cone-smoothness-with-decay} gives that 
$$
g \in H^{\tau,p}(\mR^n) \text{ if } 1 < p \leq 2,\ \alpha > n/p,\ \tau < 1/2.
$$
The assumption on $\alpha$ implies 
$$
g \in H^{s,\tilde{q}} \cap H^{s,q'}(\mR^n).
$$
Since functions in $C^s(\mR^n)$ act as pointwise multipliers on $H^{s-\eps,p}(\mR^n)$ for any $\eps > 0$ and $1 < p < \infty$ \cite[Section 4.2]{Triebel2}, we have 
$$
V \in H^{s-\eps,\tilde{q}} \cap H^{s-\eps,q'}(\mR^n).
$$
Proposition \ref{general-pointwise-sobolev-multipliers} then implies that 
$$
\norm{Vf}_{H^{s-\eps,q'}} \leq A \norm{f}_{H^{s-\eps,q}}. \qedhere
$$
\end{proof}

\begin{prop}\label{general-pointwise-sobolev-multipliers}
Let $F\in H^{\tau,\widetilde p}(\mathbb R^n)$ where $\tau\in\left[0,1\right]$ and $\widetilde p=p/(p-2)$, and $p\geq2$. Then the pointwise multiplier $T_F=f\longmapsto Ff$ maps
\[T_F\colon H^{\tau,p}(\mathbb R^n)\longrightarrow H^{\tau,p'}(\mathbb R^n)\]
continuously.
\end{prop}
\noindent For the proof of Proposition \ref{general-pointwise-sobolev-multipliers}, we need the following well-known theorem on bilinear complex interpolation.
\begin{thm}\label{bilinear-complex-interpolation}
Let $(A_0,A_1)$, $(B_0,B_1)$, $(C_0,C_1)$ be compatible Banach couples. Assume that
\[T\colon(A_0 \cap A_1)\times(B_0 \cap B_1)\longrightarrow C_0 \cap C_1\]
is bilinear and one has the bounds 
\[
\norm{T(a,b)}_{C_j} \leq M_j \norm{a}_{A_j} \norm{b}_{B_j}
\]
for $a \in A_0 \cap A_1, b \in B_0 \cap B_1, j = 0,1$. Then the operator
\[T\colon\left[A_0,A_1\right]_\vartheta\times\left[B_0,B_1\right]_\vartheta\longrightarrow\left[C_0,C_1\right]_\vartheta\]
is bounded for all $\vartheta\in\left(0,1\right)$ with norm $\leq M_0^{1-\vartheta} M_1^{\vartheta}$, where $\left[\cdot,\cdot\right]_\vartheta$ denotes the usual complex interpolation spaces.
\end{thm}
\noindent
This is a special case of e.g.\ Theorem 4.4.1 in \cite{Bergh--Lofstrom}, see also the original theorem due to Calder\'on \cite{Calderon-on-interpolation}.

\begin{proof}[Proof of Proposition \ref{general-pointwise-sobolev-multipliers}.]
We first show that if $f\in L^p(\mathbb R^n)$ and $g\in L^{\widetilde p}(\mathbb R^n)$, where $p\geq2$ and $\widetilde p=p/(p-2)$, then
\begin{equation}\label{multiplication-star}\tag{\textasteriskcentered}
fg\in L^{p'}(\mathbb R^n)\quad\text{and}\quad\bigl\|fg\bigr\|_{L^{p'}(\mathbb R^n)}\leq\bigl\|f\bigr\|_{L^p(\mathbb R^n)}\bigl\|g\bigr\|_{L^{\widetilde p}(\mathbb R^n)}.
\end{equation}
We use H\"older's inequality with $q=p/p'=p-1$ and $q'=(p-1)/(p-2)$, so that
\[\int_{\mathbb R^n}\left|fg\right|^{p'}
\leq\left(\;\int_{\mathbb R^n}\left|f\right|^{p'q}\right)^{\!\!1/q}\left(\;\int_{\mathbb R^n}\left|g\right|^{p'q'}\right)^{\!\!1/q'}.\]
Now $p'q=p$ and $p'q'=\widetilde p$ and \eqref{multiplication-star} follows.

Write $T(f,g)=fg$, so that $T_F(f)=T(f,F)$. We show that
\begin{itemize}
\item[a)] $T\colon L^p(\mathbb R^n)\times L^{\widetilde p}(\mathbb R^n)\longrightarrow L^{p'}(\mathbb R^n)$ and
\item[b)] $T\colon H^{1,p}(\mathbb R^n)\times H^{1,\widetilde p}(\mathbb R^n)\longrightarrow H^{1,p'}(\mathbb R^n)$
\end{itemize}
continuously, and the the claim of Proposition \ref{general-pointwise-sobolev-multipliers} follows from Theorem \ref{bilinear-complex-interpolation}.
But a) is just \eqref{multiplication-star} and if $f\in H^{1,p}(\mathbb R^n)$ and $g\in H^{1,\widetilde p}(\mathbb R^n)$, then we have
\[\nabla(fg)=(\nabla f)g+f(\nabla g),\]
and b) follows from \eqref{multiplication-star}.
\end{proof}

\begin{prop}\label{cone-smoothness-with-decay}
Let $c > 0$ and let $C$ be the circular cone
\[C=\bigl\{ (x',x_n) \in \mR^{n-1} \times \mR \bigm| \abs{x'}\leq cx_n \bigr\}.\]
Then $\br{x}^{-\alpha} \chi_C(x)$ belongs to $H^{\tau,p}(\mR^n)$ if $1 < p \leq 2,\ \alpha > n/p,\ \tau < 1/2$.
\end{prop}

\begin{prop}\label{cone-smoothness}
Let $c>0$ and let $C$ be the circular cone
\[C=\bigl\{ ( x',x_n ) \in\mathbb R^{n-1}\times\mathbb R\bigm|\text{$\left|x'\right|\leq cx_n$ and $x_n\leq1$}\bigr\}.\]
Then $\chi_C\in H^{\tau,p}(\mathbb R^n)$ for $\tau\in\left[0,1/2\right)$ and $p\in\left(1,2\right]$.
\end{prop}

It is useful to first consider the easier case of characteristic functions of finite straight cylinders.
\begin{prop}\label{cylinder-smoothness}
Let $C'$ be the finite straight cylinder
\[C'=\bigl\{ (x',x_n) \in\mathbb R^{n-1}\times\mathbb R\bigm|\text{$\left|x'\right|\leq1$ and $\left|x_n\right|\leq1$}\bigr\},\]
where $n\geq2$.
Then the characteristic function $\chi_{C'}$ belongs to $H^{\tau,p}(\mathbb R^n)$ for all $\tau\in\left[0,1/2\right)$ and $p\in\left(1,2\right]$.
\end{prop}

\begin{proof}
The point is that we can compute the Fourier transform of $\chi_{C'}$ explicitly. First, since $\chi_{C'}\in L^1(\mathbb R^n)$, the Fourier transform $\widehat{\chi_{C'}}$ is continuous, and there are no singularities. Thus, only the decay of $\widehat{\chi_{C'}}$ needs to be considered.

Next, using the usual radial Fourier transform (see e.g.\ Sect.\ {\S}IV.3 in \cite{Stein--Weiss}), and the fact that
\[\frac d{dx}(x^\nu\,J_\nu(x))=x^\nu\,J_{\nu-1}(x),\]
(see e.g.\ \cite{Lebedev}, Section 5.2), we can compute the Fourier transform of the characteristic function of the ball $B = B^{m}(0,1)\subset\mathbb R^{m}$ for $m \geq 2$,
\begin{align*}
\widehat{\chi_B}(\xi')&=\int_{\abs{x'} < 1} e^{-ix' \cdot \xi'} \,dx' = (2\pi)^{m/2}\int_0^1J_{m/2-1}(\left|\xi'\right|s)\,s^{m/2}\, ds\\
&=(2\pi)^{m/2}\left|\xi'\right|^{-m/2}\int_0^1(\left|\xi'\right|s)^{m/2}\,J_{m/2-1}(\left|\xi'\right|s)\, ds\\
&=(2\pi)^{m/2}\left|\xi'\right|^{-m/2}\left.\frac{(\left|\xi'\right|s)\,J_{m/2}(\left|\xi'\right|s)}{\left|\xi'\right|}\right]_0^{s=1}\\
&=(2\pi)^{m/2}\left|\xi'\right|^{-m/2}\,J_{m/2}(\left|\xi'\right|)
\end{align*}
for all $\xi'\in\mathbb R^{m}\setminus0$.
By the asymptotics of $J$-Bessel functions (see e.g.\ \cite{Lebedev}, Section 5.11), this is $\lesssim\left\langle\xi'\right\rangle^{-m/2-1/2}$ for large $\left|\xi'\right|$. Similarly, for $\xi \in\mathbb R\setminus0$, we get
\[\widehat{\chi_{\left[-1,1\right]}}(\xi)
=\frac{2\sin\xi}{\xi},\]
and this is $\lesssim\left\langle\xi\right\rangle^{-1}$ for large $\xi$.

Finally, since $\chi_{C'}(x',x_n) = \chi_{B}(x') \chi_{[-1,1]}(x_n)$ where $B$ is the unit ball in $\mR^{n-1}$, we have 
$$
\abs{\widehat{\chi_{C'}}(\xi',\xi_n)} = \abs{\widehat{\chi_B}(\xi')} \, \abs{\widehat{\chi_{[-1,1]}}(\xi_n)} \lesssim \br{\xi'}^{-n/2} \br{\xi_n}^{-1}.
$$
Thus we can estimate
\begin{align*}
\bigl\|\chi_{C'}\bigr\|_{H^{\tau,2}(\mathbb R^n)}^2
&=\int_{\mathbb R^n}\left\langle\xi\right\rangle^{2\tau}\left|\widehat{\chi_{C'}}(\xi',\xi_n)\right|^2 d\xi'\, d\xi_n\\
&\lesssim\int_{\mathbb R^{n-1}}\left\langle\xi'\right\rangle^{2\tau}\left\langle\xi'\right\rangle^{-n} d\xi'\int_{\mathbb R}\left\langle\xi_n\right\rangle^{2\tau}\left\langle\xi_n\right\rangle^{-2} d\xi_n,
\end{align*}
and this product is finite for any $\tau<1/2$, so $\chi_{C'} \in H^{\tau,2}$ for $\tau < 1/2$.

More generally, choose a fixed $\varphi \in C^{\infty}_c(\mR^n)$ such that $\chi_{C'} = \varphi \chi_{C'}$, and note that for $1 < p \leq 2$ we have by the H\"older inequality 
$$
\norm{\varphi f}_{H^{k,p}} \leq C_{\varphi} \norm{f}_{H^{k,2}}, \qquad k = 0,1.
$$
By interpolation, multiplication by $\varphi$ maps $H^{s,2}$ to $H^{s,p}$ for $0 < s < 1$, so we have $\chi_{C'} \in H^{\tau,p}$ for $\tau < 1/2$ and $1 < p \leq 2$.
\end{proof}

\begin{proof}[Proof of Proposition \ref{cone-smoothness}]
So, let $\tau<1/2$. Through a simple change of variables, it is enough to consider the cone
\[C=\bigl\{ (x',x_n) \in\mathbb R^{n-1}\times\mathbb R\bigm|\left|x'\right|\leq x_n\leq1\bigr\}.\]
We shall break this cone into pieces of the form
\[C(\alpha,\beta)=\bigl\{ (x',x_n) \in C\bigm|\alpha\leq x_n\leq\beta\bigr\},\]
where $\alpha,\beta\in\mathbb R_+$. By a simple change of variables, Prop.\ \ref{cylinder-smoothness} shows that $\chi_{C(1/2,1)}\in H^{\tau,p}(\mathbb R^n)$.
Since
\[\chi_{C(1/4,1/2)}=\chi_{C(1/2,1)}(2\cdot),\quad\chi_{C(1/8,1/4)}=\chi_{C(1/2,1)}(4\cdot),\quad\ldots,\]
we also have
\[\chi_{C(1/2,1)},\chi_{C(1/4,1/2)},\chi_{C(1/8,1/4)},\ldots\in H^{\tau,p}(\mathbb R^n),\]
and
\[\chi_C=\sum_{j=1}^\infty\chi_{C(2^{-j},2^{1-j})},\]
in the sense of distributions, and so it suffices to prove that this series converges in $H^{\tau,p}(\mathbb R^n)$. Finally, this follows from the scaling of the Sobolev norm, see Proposition \ref{sobolev-dilation} below.
\end{proof}

\begin{prop}\label{sobolev-dilation}
Let $\tau\in\left[0,1\right]$, $p\in\left(1,\infty\right)$, and let $\psi\in H^{\tau,p}(\mathbb R^n)$. Then, for $\lambda\in\left[1,\infty\right)$,
\[\bigl\|\psi(\lambda\cdot)\bigr\|_{H^{\tau,p}(\mathbb R^n)}\lesssim\lambda^{\tau-n/p}\bigl\|\psi\bigr\|_{H^{\tau,p}(\mathbb R^n)},\]
and for $\lambda\in\left(0,1\right]$,
\[\bigl\|\psi(\lambda\cdot)\bigr\|_{H^{\tau,p}(\mathbb R^n)}\lesssim\lambda^{-n/p}\bigl\|\psi\bigr\|_{H^{\tau,p}(\mathbb R^n)}.\]
\end{prop}

\begin{proof}
The result follows from complex interpolation once the cases $\tau=0$ and $\tau=1$ have been dealt with. The case $\tau=0$ follows immediately from
\[\bigl\|\psi(\lambda\cdot)\bigr\|_{L^p(\mathbb R^n)}=\lambda^{-n/p}\bigl\|\psi\bigr\|_{L^p(\mathbb R^n)}.\]
When $\tau=1$â we have
\begin{align*}
\bigl\|\psi(\lambda\cdot)\bigr\|_{H^{1,p}(\mathbb R^n)}
&\lesssim\bigl\|\psi(\lambda\cdot)\bigr\|_{L^p(\mathbb R^n)}+\sum_{j=1}^n\bigl\|\partial_j(\psi(\lambda\cdot))\bigr\|_{L^p(\mathbb R^n)}\\
&=\lambda^{-n/p}\bigl\|\psi\bigr\|_{L^p(\mathbb R^n)}
+\lambda\sum_{j=1}^n\bigl\|(\partial_j\psi)(\lambda\cdot)\bigr\|_{L^p(\mathbb R^n)}\\
&\leq\max\left\{\lambda^{-n/p},\lambda^{1-n/p}\right\}\bigl\|\psi\bigr\|_{H^{1,p}(\mathbb R^n)}. \qedhere
\end{align*}
\end{proof}

\begin{proof}[Proof of Proposition \ref{cone-smoothness-with-decay}.]
We continue to use the notation $C(\alpha,\beta)$ from the proof of Proposition \ref{cone-smoothness}. Proposition \ref{cone-smoothness} tells us that $\chi_{C(0,1)}\in H^{\tau,p}(\mathbb R^n)$, and since $\left\langle\cdot\right\rangle^{-\alpha}\in C^\infty(\mathbb R^n)$, also $\left\langle\cdot\right\rangle^{-\alpha}\chi_{C(0,1)}\in H^{\tau,p}(\mathbb R^n)$.

Thus, it suffices to prove that $\left\langle\cdot\right\rangle^{-\alpha}\chi_{C(1,\infty)}\in H^{\tau,p}(\mathbb R^n)$. We split this into series
\[\sum_{j=0}^\infty\left\langle\cdot\right\rangle^{-\alpha}\chi_{C(2^j,2^{j+1})}
=\sum_{j=0}^\infty\left\langle\cdot\right\rangle^{-\alpha}\chi_{C(1,2)}\!\left(\frac\cdot{2^j}\right),\]
and each of the individual terms belongs to $H^{\tau,p}(\mathbb R^n)$, and so it only remains to prove that the series converges in $H^{\tau,p}(\mathbb R^n)$. For this purpose, let $\varphi\in C_{\mathrm c}^\infty(\mathbb R^n)$ be such that $\varphi\equiv1$ on $\mathrm{supp}\,\chi_{C(1,2)}$ and such that $\varphi$ is supported on an $\varepsilon$-neighbourhood of $\mathrm{supp}\,\chi_{C(1,2)}$ for some $\varepsilon\in\left(0,1/4\right)$, say. Now we may estimate
\begin{align*}
&\bigl\|\left\langle\cdot\right\rangle^{-\alpha}\chi_{C(1,2)}(\cdot/2^j)\bigr\|_{H^{\tau,p}(\mathbb R^n)}
=\bigl\|\left\langle\cdot\right\rangle^{-\alpha}\varphi(\cdot/2^j)\,\chi_{C(1,2)}(\cdot/2^j)\bigr\|_{H^{\tau,p}(\mathbb R^n)}\\
&\lesssim\bigl\|\left\langle\cdot\right\rangle^{-\alpha}\varphi(\cdot/2^j)\bigr\|_{C^1(\mathbb R^n)}\bigl\|\chi_{C(1,2)}(\cdot/2^j)\bigr\|_{H^{\tau,p}(\mathbb R^n)}
\lesssim 2^{-\alpha j}\cdot2^{jn/p},
\end{align*}
and since the exponent $-\alpha + n/p$ is negative, the series in question is comparable to geometric series and converges, as required.
\end{proof}

\begin{proof}[Proof of Theorem \ref{thm_cgo}]
We look for a solution of $(-\Delta+V-\lambda)u = 0$ of the form $u = e^{i\zeta \cdot x}(1+\psi)$ where $\zeta \in \mC^n$ satisfies $\zeta \cdot \zeta = \lambda$. Now $u$ will be a solution if and only if $\psi$ satisfies 
$$
(-\Delta + 2\zeta \cdot D + V)\psi = -V.
$$

We wish to use Propositions \ref{prop_cgo1} and \ref{prop_cone_pointwisemultiplier} to solve this equation. Note that Proposition \ref{prop_cone_pointwisemultiplier} has the assumption $r = \frac{2(n+1)}{n+3}$, which implies that 
$$
\frac{1}{r} - \frac{1}{r'} = \frac{2}{n+1}.
$$
This is consistent with Proposition \ref{prop_cgo1}, which requires that 
$$
\frac{1}{r} - \frac{1}{r'} \in \left[\frac{2}{n+1}, \frac{2}{n}\right).
$$
We have chosen $r$ so that $1/r-1/r'$ is as small as possible, in order to arrange the best power of $\abs{\re(\zeta)}$ in the estimates.

Note that one has 
$$
(r, \ r', \ 1/r-1/r') = \left\{ \begin{array}{cl} (6/5, \ 6, \ 2/3) & \text{if } n = 2, \\ (4/3, \ 4, \ 1/2) & \text{if }n=3. \end{array} \right.
$$
By Proposition \ref{prop_cone_pointwisemultiplier}, the function $V$ satisfies the condition in Proposition \ref{prop_cgo1} with $s$ replaced by some $t<s$, where $t > 0$ if $n=2$ and $t > 1/4$ if $n=3$. If $\abs{\re(\zeta)}$ is sufficiently large, we obtain a solution $\psi$ satisfying 
$$
\norm{\psi}_{H^{t,r'}} \leq C \abs{\re(\zeta)}^{n(1/r-1/r')-2}\norm{V}_{H^{t,r}} \leq C \abs{\re(\zeta)}^{-2/(n+1)}.
$$
We have the Sobolev embedding $H^{t,r'} \subset L^q$ where $q=\frac{nr'}{n-tr'}>r'$ (here we assume $t < n/r'$). Thus we have 
$$
\norm{\psi}_{L^q} \leq C \abs{\re(\zeta)}^{-\frac{2}{n+1}} =  C \abs{\re(\zeta)}^{-\frac{n}{q}-\delta}
$$
where $\delta = \frac{2}{n+1}-\frac{n}{q} = \frac{2}{n+1}-\frac{n-tr'}{r'} > 0$ by our assumptions on $r'$ and $t$. The result follows by taking $\rho = -i\zeta$.
\end{proof}

\section{Reduction to Laplace transform} \label{sec_reduction_laplace}

We also need to analyze further the solution $v_0$ in Proposition \ref{prop:exponential_orthogonality}. As a solution of $(-\Delta-\lambda)v_0 = 0$ in $\mR^n$, $v_0$ is real-analytic and has a Taylor series at the origin. If $v_0$ is nonzero, the Taylor series is not identically zero. Assume that the first nonvanishing term has degree $N$, and denote by $H(x)$ the homogeneous polynomial of all terms of degree $N$. Thus 
$$
v_0(x) = H(x) + R(x), \qquad \abs{R(x)} \leq C \abs{x}^{N+1},
$$
for $x$ near the origin.
The next observation is \cite[Lemma 3.4]{BPS}.

\begin{lemma} \label{lemma_h_harmonic}
If $v_0 \in B^*$ is a nontrivial solution of $(-\Delta-\lambda)v_0 = 0$, then the lowest degree nonvanishing terms $H$ in the Taylor expansion of $v_0$ form a harmonic homogeneous polynomial in $\mR^n$.
\end{lemma}

The main point in this section is the following result, proved by using the solutions of Theorem \ref{thm_cgo} in Proposition \ref{prop:exponential_orthogonality}. This result implies that whenever $\lambda$ is a non-scattering energy, then the Laplace transform of $\chi_C H$ vanishes in a certain complex manifold.

\begin{prop} \label{prop_laplace_vanishing}
Let $\lambda > 0$, let $n \in \{2, 3\}$, and let $V(x) = \chi_{C}(x) \varphi(x)$ where $\varphi \in \left\langle\cdot\right\rangle^{-\alpha}C^s(\mR^n)$ is superexponentially decaying with $\alpha>5/3$ and $s>0$ for $n=2$ (resp.\ $\alpha>9/4$ and $s>1/4$ for $n=3$), $\varphi(0)\neq0$, and $C \subset \mR^n$ is a closed circular cone opening in direction $e_n$ with angle $< \pi$ having vertex at the origin. Let $U$ be a neighborhood of $e_n$ in $S^{n-1}$ such that $e^{-\tau \omega \cdot x}$ is exponentially decaying in $C$ whenever $\tau > 0$ and $\omega \in \overline U$.

Assume that $\lambda$ is a non-scattering energy for $V$, let $v_0$ be the solution in Proposition \ref{prop:exponential_orthogonality}, and write $v_0 = H + R$ where $H$ is a harmonic homogeneous polynomial of degree $N$ as in Lemma \ref{lemma_h_harmonic}. Then 
$$
\int_{C} e^{-\rho \cdot x} H(x) \,dx = 0, \qquad \rho \in W_{0},
$$
where $W_{\lambda}$ is the set of all $\rho \in \mC^n$ of the form  
\begin{equation*} 
\rho = \rho(\tau, \omega, \omega') = \tau \omega + i(\tau^2+\lambda)^{1/2}\omega'
\end{equation*}
and where the parameters satisfy 
\begin{equation*} 
\tau > M, \quad \omega \in U, \quad \omega' \in S^{n-1} \text{ with } \omega \cdot \omega' = 0
\end{equation*}
for some sufficiently large $M$.
\end{prop}


Notation: All implicit constants below are allowed to depend on all the parameters except for $\rho$.

\begin{proof}
Let $\rho\in W_\lambda$. Theorem \ref{thm_cgo} guarantees the existence of a CGO solution $u=e^{-\rho\cdot x}\left(1+\psi\right)$, where $\psi$ depends on $\rho$, with 
\[\bigl\|\psi\bigr\|_{L^q(\mathbb R^n)}\lesssim\frac1{\left|\rho\right|^c},\]
where $q=2(n+1)/(n-1)$ and $c=n(n-1)/2(n+1)+\delta$ for some small $\delta\in\mathbb R_+$. To simplify notation, we assume, as we may, that $\delta<s$.
We also remark that $\left(V(x)-1\right)/(\left|x\right|^s)$ is bounded in $C$, where for simplicity we have assumed that $\varphi(0) = 1$.

We define $F(\rho)$ to be the Laplace transform 
\[
F(\rho) := \int_C e^{-\rho \cdot x} H(x) \,dx, \qquad \re(\rho) \in U.
\]
Our goal is to prove that
\begin{equation} \label{frho_claim}
F(\rho)=0, \qquad \rho \in W_0.
\end{equation}
To do this, we observe that Proposition \ref{prop:exponential_orthogonality} yields 
\[
\int_{\mR^n} V e^{-\rho \cdot x} (1+\psi)(H+R) \,dx = 0, \qquad \rho \in W_{\lambda},
\]
which can be rewritten as 
\begin{equation} \label{frho_rewritten}
F(\rho) = -\int_{C} e^{-\rho \cdot x} \big[ (V-1)H + V(R + \psi(H+R)) \big] \,dx, \ \ \rho \in W_{\lambda}.
\end{equation}
By the homogeneity of $H(x)$ and using that $C$ is a cone, the left hand side satisfies 
\[F(\rho) 
=\left|\rho\right|^{-N-n} F(\rho/\abs{\rho}).\]
For the right hand side, we claim that for all $\rho \in W_{\lambda}$ one has 
\begin{equation}\label{extra-decay}
\left\lvert \int_{C} e^{-\rho \cdot x} \big[ (V-1)H + V(R + \psi(H+R)) \big] \,dx \right\rvert \lesssim \left|\rho\right|^{-N-n-\delta}.
\end{equation}

Assuming that \eqref{extra-decay} holds, \eqref{frho_rewritten} implies 
\[F\!\left(\frac\rho{\left|\rho\right|}\right)\lesssim\left|\rho\right|^{-\delta},\]
which holds for $\rho\in W_\lambda$, and more precisely,
\[F\!\left(\frac{\tau\omega}{\sqrt{2\tau^2+\lambda}}+\frac{i\omega'\sqrt{\tau^2+\lambda}}{\sqrt{2\tau^2+\lambda}}\right)\lesssim\left(\kern-1pt\sqrt{2\tau^2+\lambda}\kern1pt\right)^{\!-\delta}\lesssim\tau^{-\delta}\]
for $\tau\gg1$, $\omega\in U$ and $\omega'\in S^{n-1}$ with $\omega'\perp\omega$. Now, taking $\tau\longrightarrow\infty$ gives
\[F\!\left(\frac\omega{\sqrt2}+\frac{i\omega'}{\sqrt2}\right)=0,\]
which holds for all $\omega\in U$ and all $\omega'\in S^{n-1}$ with $\omega\perp\omega'$. By homogeneity we have
\[F\!\left(t\omega+it\omega'\right)=0\]
for all $t\in\mathbb R_+$, $\omega\in U$ and $\omega'\in S^{n-1}$ with $\omega\perp\omega'$, which proves \eqref{frho_claim} as required.

It remains to show \eqref{extra-decay}. We shall split the left hand side of \eqref{extra-decay} into many integrals which are easier to estimate. The following is essentially Lemma 3.6 from \cite{BPS}.

\begin{lemma}\label{holder-lemma}
Let $R\colon\mathbb R^n\setminus\{0\}\longrightarrow\mathbb C$ be a continuous homogeneous function of degree $N$, and let $e^{-\mathrm{Re}(\rho) \cdot x}$ be exponentially decaying in $C$. Then, for any $f\in L^q(\mathbb R^n)$, where $q\in\left[1,\infty\right)$, we have
\[\int_Ce^{-\rho \cdot x}R(x)f(x)\,dx\lesssim\left|\rho\right|^{n/q-N-n}\bigl\|e^{-(\rho/\left|\rho\right|)\cdot x}R\bigr\|_{L^{q'}(\mathbb R^n)}\bigl\|f\bigr\|_{L^q(\mathbb R^n)}.\]
\end{lemma}
\noindent This follows immediately from the change of variables $y=x/\left|\rho\right|$ and H\"older's inequality.

First we observe that, by H\"older's inequality,
\begin{align*}
&\int_{C\setminus B(0,\varepsilon)}e^{-\rho\cdot x}\left(1+\psi(x)\right)V(x)\,v_0(x)\, dx\\
&\qquad\lesssim e^{-\varepsilon d\left|\rho\right|}\left(\bigl\|Vv_0\bigr\|_{L^1(\mathbb R^n)}
+\bigl\|\psi\bigr\|_{L^q(\mathbb R^n)}\bigl\|Vv_0\bigr\|_{L^{q'}(\mathbb R^n)}\right)
\lesssim e^{-\varepsilon d\left|\rho\right|},
\end{align*}
where $d$ is some suitably small positive real constant.
We also have
\[\int_{C\setminus B(0,\varepsilon)}e^{-\rho\cdot x}\,H(x)\, dx
\lesssim e^{-\varepsilon d\left|\rho\right|}.\]

Also, Lemma \ref{holder-lemma} gives, observing that $n(n-1)/2(n+1)-c=-\delta$, the estimates
\begin{align*}
&\int_{C\cap B(0,\varepsilon)}e^{-\rho\cdot x}
H(x)\left(V(x)-1\right) \,dx\\
&\qquad=\int_{C\cap B(0,\varepsilon)}e^{-\rho\cdot x}H(x)\left|x\right|^s\frac{V(x)-1}{\left|x\right|^s}\, dx\lesssim\left|\rho\right|^{-N-n-s},\\
&\int_{C\cap B(0,\varepsilon)}e^{-\rho\cdot x}\,\psi(x)\,V(x)\,H(x)\, dx\lesssim\left|\rho\right|^{-N-n-\delta},\\
&\int_{C\cap B(0,\varepsilon)}e^{-\rho\cdot x}\,\psi(x)\,V(x)\,O(\left|x\right|^{N+1})\, dx \lesssim \left|\rho\right|^{-N-n-1-\delta},\quad\text{and}\\
&\int_{C\cap B(0,\varepsilon)}e^{-\rho\cdot x}\,V(x)\,O(\left|x\right|^{N+1})\, dx
\lesssim\left|\rho\right|^{-N-n-1}.
\end{align*}
Combining the above estimates gives the desired claim \eqref{extra-decay}.
\end{proof}

\section{The two-dimensional case} \label{sec_two_dimensions}

We now restrict our attention to the two-dimensional case and prove Theorem \ref{thm_mainthm1}. Assume for the sake of contradiction that $\lambda > 0$ is a non-scattering energy for a potential $V$ as in Theorem \ref{thm_mainthm1}. By Proposition \ref{prop_laplace_vanishing}, this implies the vanishing of the following Laplace transform for a nonzero homogeneous harmonic polynomial $H$ of degree $N$, 
\[\int_Ce^{-(\omega+i\omega')\cdot x}\,H(x)\, dx=0\]
for all $\omega\in U$ where $U$ is an open subset of the unit circle $S^1$, and for all $\omega'\in S^1$ is such that $\omega\perp\omega'$. Since we are in two dimensions, $H$ must be of the form 
\[H(x)=a\,(x_1+ix_2)^N+b\,(x_1-ix_2)^N\]
for some constants $a$ and $b$, when $N>0$. If $N=0$, then $H(x)$ will be just a constant $a$. The goal is to prove that $a=b=0$, which will contradict the fact that $H$ is nonzero and will prove Theorem \ref{thm_mainthm1}.

First we introduce some notation. For simplicity, we assume that the cone $C$ opens in direction $e_1$ instead of $e_2$ as in Proposition \ref{prop_laplace_vanishing}. We shall write $S$ for the arc $S^1\cap C$, and $I$ for the interval $\left[-L/2,L/2\right]$ parametrizing $S$ under the mapping $r\longmapsto e^{ir}\colon\left(-\pi,\pi\right]\longrightarrow S^1$. Here $L\in\left(0,\pi\right)$ is the opening angle of the sector $C$. Now we can take $U$ to be the arc of $S^1$ corresponding, in the same coordinates, to the interval $\left(-\pi/2+L/2,\pi/2-L/2\right)$.

In polar coordinates, we have
\[\int_S\int_0^\infty e^{-(\omega+i\omega')\cdot\vartheta\,r}\,r^{N+1}\, dr\,H(\vartheta)\, d\vartheta=0.\]
We wish to rewrite the $r$-integral. Let $\alpha\in\mathbb C$ have a positive real part (we will take $\alpha = (\omega + i \omega') \cdot \vartheta$). Then, by Cauchy's integral theorem, as the integrand is exponentially decaying in the right half-plane of the complex plane, we can rotate the path of integration from the half-ray $\left\{\alpha r\,\middle|\,r\in\mathbb R_+\right\}$ to $\mathbb R_+$ giving
\[\int_0^\infty e^{-\alpha r}\left(\alpha r\right)^{N+1}\alpha\, dr
=\int_0^\infty e^{-r}\,r^{N+1}\, dr=\Gamma(N+2)=(N+1)!.\]
Thus we obtain 
\[\int_S\left(\left(\omega+i\omega'\right)\cdot\vartheta\right)^{-N-2}
H(\vartheta)\, d\vartheta=0.\]

With the parametrization $\omega=e^{i\varphi}$ and $\vartheta=e^{i\psi}$, we can compute
\begin{align*}
\left(\omega+i\omega'\right)\cdot\vartheta
&=\left[\begin{array}{c}
\cos\varphi\mp i\sin\varphi\\\sin\varphi\pm i\cos\varphi\end{array}\right]
\cdot\left[\begin{array}{c}
\cos\psi\\\sin\psi\end{array}\right]
=\left[\begin{array}{c}e^{\mp i\varphi}\\\pm ie^{\mp i\varphi}\end{array}\right]
\cdot\left[\begin{array}{c}\cos\psi\\\sin\psi\end{array}\right]\\
&=e^{\mp i\varphi}\left(\cos\psi\pm i\sin\psi\right)
=e^{\mp i\varphi}e^{\pm i\psi}
=e^{\mp i\left(\varphi-\psi\right)}
=e^{\pm i\left(\psi-\varphi\right)}.
\end{align*}
Thus, the expression involving $\omega$ and $\omega'$ factors nicely and the variables $\varphi$ and $\psi$ become separated.

In the case $N=0$ we then have
\[a\int_Se^{\mp i2\psi} dx=0,\]
or more simply
\[a\,\widehat{\chi_I}(\pm2)=0.\]
The Fourier coefficient is easy to compute and we get
\[a\sin L=0,\]
and so we must have $H(x)\equiv a=0$.

When $N>0$ we get
\[\int_Se^{\mp i(N+2)\psi}\left(ae^{iN\psi}+be^{-iN\psi}\right) d\psi=0.\]
This leads to the pair of equations
\[\left\{\begin{array}{lll}
a\int_{-L/2}^{L/2}e^{-2i\psi}\, d\psi
&+b\int_{-L/2}^{L/2}e^{-i(2N+2)\psi}\, d\psi&=0,\\
a\int_{-L/2}^{L/2}e^{i(2N+2)\psi}\, d\psi
&+b\int_{-L/2}^{L/2}e^{2i\psi}\, d\psi&=0.
\end{array}\right.\]
In terms of Fourier coefficients this reads
\[\left\{\begin{array}{lll}
a\,\widehat{\chi_I}(2)&+b\,\widehat{\chi_I}(2N+2)&=0,\\
a\,\widehat{\chi_I}(2N+2)&+b\,\widehat{\chi_I}(2)&=0,
\end{array}\right.\]
where we have used the fact that $\chi_I$ is even.
This is a homogeneous linear system of equations for $a$ and $b$, and if the determinant of the coefficient matrix is nonzero, then we must have $a=b=0$. The determinant can vanish only if
\[\widehat{\chi_I}(2)=\pm\widehat{\chi_I}(2N+2).\]
The Fourier coefficients of $\chi_I$ are easy to compute, and the vanishing of the determinant simplifies to
\[\sin L=\pm\frac1{N+1}\sin\left((N+1)L\right).\]

It is now straightforward to check that this equation has no solutions $L$ in the interval $\left(0,\pi\right)$. The derivative
\[\frac{ d}{ dL}\left(\sin L\mp\frac1{N+1}\sin\left((N+1)L\right)\right)
=\cos L\mp\cos\left((N+1)L\right)\]
is clearly positive when $0<L<\pi/(N+1)$. When \[L\in\left[\pi/(N+1),N\pi/(N+1)\right],\] we clearly have
\[\sin L>\frac1{N+1}\geq\frac1{N+1}\sin\left((N+1)L\right).\]
Finally, the case where $L$ belongs to $\left(N\pi/(N+1),\pi\right)$ reduces to the case where $L$ belongs to $\left(0,\pi/(N+1)\right)$ by the change of variables $L\longmapsto\pi-L$.

\section{The three-dimensional case} \label{sec_three_dimensions}

By the same argument as in the beginning of Section \ref{sec_two_dimensions}, the proof of Theorem \ref{thm_mainthm2} reduces to showing the following result.

\begin{lemma} \label{lemma_sphericalharmonic_condition}
Let $n = 3$, and let $S_{\gamma} = \{ x \in S^{n-1} \,;\, x_n > \cos \gamma \}$ be a spherical cap where $0 < \gamma < \pi/2$ . There is a countable subset $E \subset (0,\pi/2)$ such that for any $\gamma \in (0,\pi/2) \setminus E$, the condition 
$$
\int_{S_{\gamma}} ((e_n + i\eta) \cdot x)^{-N-n} H(x) \,dx = 0, \qquad \eta \in S^{n-1}, \ \ \eta \cdot e_n = 0,
$$
implies that $H \equiv 0$ whenever $H$ is a spherical harmonic on $S^{n-1}$ of degree $N$.
\end{lemma}

To prepare for the proof, write $x = ((\sin \alpha) \omega', \cos \alpha)$ where $\omega' \in S^{n-2}$. Writing also $\eta = (\eta',0)$ where $\eta' \in S^{n-2}$, the integral becomes 
\begin{multline*}
\int_{S_{\gamma}} ((e_n + i\eta) \cdot x)^{-N-n} H(x) \,dx \\
 = \int_0^{\gamma} \int_{S^{n-2}} (\cos \alpha + i (\sin \alpha) \eta' \cdot \omega')^{-N-n} H((\sin \alpha) \omega', \cos \alpha) \sin^{n-2} \alpha \,d\omega' \,d\alpha.
\end{multline*}
Let $\{ Y_1^N, \ldots, Y_r^N \}$ be some basis of spherical harmonics of degree $N$ where $r = r_N$, and write $H = \sum_{j=1}^r a_j Y_j^N$. It is convenient to rephrase this in terms of rotation matrices: we write $\eta' = R e_1$ where $R$ is a rotation matrix, that is, $R \in SO(n-1)$. The condition in Lemma \ref{lemma_sphericalharmonic_condition} then becomes 
\begin{equation} \label{aj_fjn_condition}
\sum_{j=1}^r a_j f_j^N(\gamma; R) = 0, \qquad R \in SO(n-1),
\end{equation}
where 
\begin{multline*}
f_j^N(\gamma; R) := \\
 \int_0^{\gamma} \int_{S^{n-2}} (\cos \alpha + i (\sin \alpha) \omega'_1)^{-N-n} Y_j^N((\sin \alpha) R \omega', \cos \alpha) \sin^{n-2} \alpha \,d\omega' \,d\alpha.
\end{multline*}
Here we have changed variables $\omega' \mapsto R \omega'$ in the integral (note that this uses the fact that $S_{\gamma}$ is a spherical cap).

The next result together with an analyticity argument will imply Lemma \ref{lemma_sphericalharmonic_condition}.

\begin{lemma} \label{lemma_notidenticallyzero}
Assume that $n=3$. For any $N \geq 0$, there exists a basis $\{ Y_1^N, \ldots, Y_r^N \}$ of spherical harmonics of degree $N$ and there exist rotation matrices $R_1, \ldots, R_r \in SO(n-1)$ such that the function 
$$
g^N: (0,\pi/2) \to \mC, \ \ g^N(\gamma) := \det\left[ (f_j(\gamma; R_k))_{j,k=1}^r \right]
$$
is not identically zero.
\end{lemma}
\begin{proof}
Assume that $n=3$. In this case there is an explicit basis of spherical harmonics of degree $N$ given by 
$$
Y_j^N((\sin \alpha) \omega', \cos \alpha) = P_N^{\abs{j}}(\cos \alpha) e^{ij \beta}, \quad -N \leq j \leq N
$$
where $P_N^m$ are associated Legendre polynomials and $\omega' = (\cos \beta, \sin \beta)$. (As is customary, we index the basis by $-N \leq j \leq N$ instead of $1 \leq j \leq 2N+1$.) Let $R_k$ be the rotation in $S^1$ by angle $\theta_k$. Then 
$$
Y_j^N((\sin \alpha) R_k \omega', \cos \alpha) = e^{ij \theta_k} Y_j^N((\sin \alpha) \omega', \cos \alpha).
$$
This implies that 
$$
f_j^N(\gamma; R_k) = e^{ij \theta_k} f_j(\gamma)
$$
where $f_j(\gamma) := f_j^N(\gamma;\mathrm{Id})$, and 
$$
g^N(\gamma) = f_1(\gamma) \cdots f_r(\gamma) \det\left[ (e^{ij\theta_k})_{j,k=-N}^N \right].
$$
The last determinant is of Vandermonde type. We choose the rotations so that $e^{i \theta_k} \neq e^{i\theta_l}$ for $k \neq l$, and then the last determinant is nonzero.

To show that $g^N(\gamma)$ is not identically zero, we need to demonstrate that there is some $\gamma \in (0,\pi/2)$ such that the product $f_1(\gamma) \cdots f_r(\gamma)$ is nonzero. We first prove that none of the functions $f_j$ is identically zero in $(0,\pi/2)$. Now 
\begin{multline*}
f_j(\gamma) = \\
 \int_0^{\gamma} \int_{S^1} (\cos \alpha + i (\sin \alpha) \omega'_1)^{-N-n} Y_j^N((\sin \alpha) \omega', \cos \alpha) \sin \alpha \,d\omega' \,d\alpha.
\end{multline*}
Each $f_j$ extends analytically near $\left[0,\pi/2\right)$, and its derivative satisfies  
$$
f_j'(\gamma) = \sin \gamma \int_{S^1} (\cos \gamma + i (\sin \gamma) \omega'_1)^{-N-n} Y_j^N((\sin \gamma) \omega', \cos \gamma) \,d\omega'.
$$
Inserting the explicit form for $Y_j^N$ we get 
$$
f_j'(\gamma) = P_N^{\abs{j}}(\cos \gamma) \sin \gamma \int_0^{2\pi} (\cos \gamma + i \sin \gamma \cos \beta)^{-N-n} e^{ij \beta} \,d\beta.
$$
It is enough to show that the function $\gamma \mapsto \int_0^{2\pi} \cdots \,d\beta$ is not identically zero. For $j = 0$ this follows just by taking $\gamma=0$, and for $j \neq 0$ the result follows by differentiating $\abs{j}$ times with respect to $\gamma$ and taking $\gamma=0$. More precisely, writing $p=\cos\gamma+i\sin\gamma\cos\beta$ and $p'= dp/ d\gamma=-\sin\gamma+i\cos\gamma\cos\beta$, the $\left|j\right|$th derivative of the integral has the form
\[\int_0^{2\pi}\frac{\nu_0(p')^{\left|j\right|}+\nu_1p(p')^{\left|j\right|-1}+\ldots+\nu_{\left|j\right|}p^{\left|j\right|}}{p^{N+n+\left|j\right|}}\,e^{ij\beta}\, d\beta,\]
for some constants $\nu_0,\nu_1,\ldots,\nu_{\left|j\right|}\in\mathbb C$, and in particular, the coefficient $\nu_0$ is
\[\nu_0=\pm(N+n)(N+n+1)\cdots(N+n+\left|j\right|-1)\neq0.\]
At $\gamma=0$, we have $p=1$ and $p'=i\cos\beta$, and the integral simplifies to
\[\int_0^{2\pi}\left(\nu_0'\cos^{\left|j\right|}\beta+\nu_1'\cos^{\left|j\right|-1}\beta+\ldots+\nu_{\left|j\right|}'\right)e^{ij\beta}\, d\beta,\]
where the coefficients $\nu_0'$, $\nu_1'$, \dots are the same coefficients as before except for the obvious powers of $i$.
Writing the cosines in terms of exponentials, there will be exactly one term which resonates with $e^{ij\beta}$, namely the exponential $e^{-ij\beta}$ coming from $\cos^{\left|j\right|}\beta$, and its coefficient is nonzero. Thus the $\left|j\right|$th derivative of $\int_0^{2\pi}\ldots d\beta$ at $\gamma=0$ is nonzero, and as an analytic function of $\gamma$, the integral can not be identically zero.

We have proved that each $f_j$ is not identically zero, and since $f_j$ extends analytically near $[0,\pi/2)$ it is nonvanishing in $(0,\pi/2) \setminus E_j$ for some countable discrete subset $E_j \subset (0,\pi/2)$. Then $f_1 \cdots f_r$ is nonvanishing in $(0,\pi/2) \setminus E$ where $E = \cup_{j=1}^r E_j$ is a countable set.
\end{proof}

\begin{proof}[Proof of Lemma \ref{lemma_sphericalharmonic_condition}]
Each function $\gamma \mapsto f_j^N(\gamma,R)$ extends as an analytic function in some neighborhood of the interval $[0,\pi/2)$ in the complex plane, and the same is true for the functions $g^N$ in Lemma \ref{lemma_notidenticallyzero}. For each $N$, by Lemma \ref{lemma_notidenticallyzero} we can choose $Y_j^N$ and $R_j$ such that $g^N$ is analytic in some neighborhood $U_N$ of $[0,\pi/2)$ and $g^N|_{(0,\pi/2)}$ is not identically zero. By analyticity the set $E_N = \{ z \in U_N \,;\, g^N(z) = 0 \}$ is countable and discrete in $U_N$.

Define 
$$
E = \bigcup_{N=0}^{\infty} (E_N \cap (0,\pi/2)).
$$
Then $E$ is a countable subset of $(0,\pi/2)$, and each $g^N$ is nonvanishing in $(0,\pi/2) \setminus E$.

Assume now that $\gamma \in (0,\pi/2) \setminus E$, let $N \geq 0$, and let $H$ be a spherical harmonic of degree $N$ such that the condition in Lemma \ref{lemma_sphericalharmonic_condition} holds. Writing $H = \sum_{j=1}^r a_j Y_j^N$ where $Y_j^N$ and $R_j$ were chosen above, by \eqref{aj_fjn_condition} we have 
$$
\sum_{j=1}^r a_j f_j^N(\gamma; R_k) = 0, \qquad k = 1, \ldots, r.
$$
But $g^N(\gamma) \neq 0$ so the matrix $(f_j^N(\gamma; R_k))_{j,k=1}^r$ is invertible, which implies that $a_j = 0$ for $j=1, \ldots, r$. This proves that $H \equiv 0$.
\end{proof}

\bibliographystyle{alpha}

\end{document}